\documentclass{amsart}
\usepackage{amsfonts}

\newtheorem{theorem}{Theorem}
\theoremstyle{plain}

\numberwithin{equation}{section}
\input{tcilatex}

\begin{document}
\title[The Properties of Bertrand Curves in Dual Space]{The Properties of
Bertrand Curves in Dual Space}
\author{\.{I}lkay ARSLAN G\"{U}VEN}
\address{ }
\email{}
\urladdr{}
\author{\.{I}pek A\u{G}AO\u{G}LU}
\curraddr{ }
\email{ }
\urladdr{}
\date{}

\begin{abstract}
In this study, we investigate Bertrand curves in three dimensional dual
space $\mathbb{D}^{3}$ and we obtain the characterizations of these curves
in dual space $\mathbb{D}^{3}$. Also we show that involutes of a curve
constitute Bertrand pair curves.
\end{abstract}

\maketitle

\section{Introduction}

In the study of differential geometry, the characterizations of the curves
and the corrsponding relations between the curves are significant problem.
It is well known that many important results in the theory of the curves in $%
E^{3}$ were given by G. Monge and then G. Darboux detected the idea of
moving frame. After this Frenet defined moving frame and special equations
which are used in mechanics, kinematics and physics.

A set of orthogonal unit vectors can be built, if a curve is differentiable
in an open interval, at each point. These unit vectors are called Frenet
frame. The Frenet vectors along the curve, define curvature and torsion of
the curve. The frame vectors, curvature and torsion of a curve constitute
Frenet apparatus of the curve.

It is certainly well known that a curve can be explained by its curvature
and torsion except as to its position in space. The curvature ($\kappa )$
and torsion ($\tau )$ of a regular curve help us to specify the shape and
size of the curve. Such as; If $\kappa =\tau =0$, then the curve is a
geodesic. If $\kappa \neq 0$ (constant) and $\tau =0$, then the curve is a
circle with radius $\frac{1}{\kappa }$. If $\kappa \neq 0$ (constant) and $%
\tau \neq 0$ (constant), then the curve is a helix.

Bertrand curves can be given as another example of that relation. Bertrand
curves are discovered in 1850, by J. Bertrand who is known for his
applications of differential equations to physics, especially
thermodynamics. A Bertrand curve in $E^{3}$ is a curve such that its
principal normal vectors are the principal normal vectors of an other curve.
It is proved in most studies on the subject that the characteristic property
of a Bertrand curve is the existence of a linear relation between its
curvature and torsion as:%
\begin{equation*}
\lambda \kappa +\mu \tau =1
\end{equation*}%
with constants $\lambda ,$ $\mu $ where $\lambda \neq 0$ (see \cite{Kuh}).

Dual numbers were defined by W.K.Clifford (1849-1879). After him E. Sudy
used dual numbers and dual vectors to clarify a mapping from dual unit
sphere to three dimensional Euclidean space $E^{3}$. This mapping is called
Study mapping. Study mapping corresponds the dual points of a dual unit
sphere to the oriented lines in $E^{3}.$ So the set of oriented lines in
Euclidean space $E^{3}$ is one to one correspondence with the points of dual
space in $\mathbb{D}^{3}.$

In this paper, we study Bertrand curves in dual space $\mathbb{D}^{3}.$

\section{Preliminaries}

We now recall some basic notions about dual space and apparatus of curves.

The set $\mathbb{D}$ is called the dual number system and the elements of
this set are in type of $\widehat{a}=a+\varepsilon a^{\ast }$. Here $a$ and $%
a^{\ast }$ are real numbers and $\varepsilon ^{2}=0$ which is called a dual
unit. The elements of the set $\mathbb{D}$ are called dual numbers. The set $%
\mathbb{D}$ is given by 
\begin{equation*}
\mathbb{D}=\left\{ \widehat{a}=a+\varepsilon a^{\ast }\mid a,a^{\ast
}\epsilon \text{ }\mathbb{R}\right\} .
\end{equation*}%
For the dual number $\widehat{a}=a+\varepsilon a^{\ast },$ $a\in \mathbb{R}$
is called the real part of $\widehat{a}$ and $a^{\ast }\in \mathbb{R}$ is
called the dual part of $\widehat{a}.$

Two inner operations and equality on $\mathbb{D}$ are defined for $\widehat{a%
}$ $=a+\varepsilon a^{\ast }$ and $\widehat{b}=b+\varepsilon b^{\ast }$, as ;

1) $+:\mathbb{D\times D\longrightarrow D}$%
\begin{equation*}
\widehat{a}+\widehat{b}=(a+\varepsilon a^{\ast })+(b+\varepsilon b^{\ast
})=(a+b)+\varepsilon (a^{\ast }+b^{\ast })
\end{equation*}%
is called the addition in $\mathbb{D}.$

\bigskip 2) $\cdot :\mathbb{D\times D\longrightarrow D}$%
\begin{equation*}
\widehat{a}\cdot \widehat{b}=(a+\varepsilon a^{\ast }).(b+\varepsilon
b^{\ast })=a.b+\varepsilon (ab^{\ast }+ba^{\ast })
\end{equation*}%
is called the multiplication in $\mathbb{D}.$

3) $\widehat{a}=\widehat{b}$ if and only if $a=b$ and $a^{\ast }=b^{\ast }.$%
(see\cite{Kos,Vel})

Also the set $\mathbb{D}=\left\{ \widehat{a}=a+\varepsilon a^{\ast }\mid
a,a^{\ast }\epsilon \text{ }\mathbb{R}\right\} $ forms a commutative ring
with the following operations%
\begin{eqnarray*}
i)\text{ \ }(a+\varepsilon a^{\ast })+(b+\varepsilon b^{\ast })
&=&(a+b)+\varepsilon (a^{\ast }+b^{\ast }) \\
ii)\text{ \ }(a+\varepsilon a^{\ast }).(b+\varepsilon b^{\ast })
&=&a.b+\varepsilon (ab^{\ast }+ba^{\ast }).
\end{eqnarray*}%
The division of two dual numbers $\widehat{a}=a+\varepsilon a^{\ast }$ and $%
\widehat{b}=b+\varepsilon b^{\ast }$ provided $b\neq 0$ can be defined as%
\begin{equation*}
\frac{\widehat{a}}{\widehat{b}}=\frac{a+\varepsilon a^{\ast }}{b+\varepsilon
b^{\ast }}=\frac{a}{b}+\varepsilon \frac{a^{\ast }b-ab^{\ast }}{b^{2}}.
\end{equation*}%
The set%
\begin{equation*}
\mathbb{D}^{3}=\mathbb{D}\times \mathbb{D}\times \mathbb{D}=\left\{ 
\begin{array}{c}
\overrightarrow{\widehat{a}}\mid \overrightarrow{\widehat{a}}=\left(
a_{1}+\varepsilon a_{1}^{\ast },a_{2}+\varepsilon a_{2}^{\ast
},a_{3}+\varepsilon a_{3}^{\ast }\right) \\ 
=\left( a_{1},a_{2},a_{3}\right) +\varepsilon \left( a_{1}^{\ast
},a_{2}^{\ast },a_{3}^{\ast }\right) \\ 
=\overrightarrow{a}+\varepsilon \overrightarrow{a^{\ast }},\text{ \ \ }%
\overrightarrow{a}\in \mathbb{R}^{3}\text{ },\text{ }\overrightarrow{a^{\ast
}}\in \mathbb{R}^{3}%
\end{array}%
\right\}
\end{equation*}%
is a module on the ring $\mathbb{D}$ which is called $\mathbb{D}$-Module and
the elements are dual vectors consisting of two real vectors.

The inner product and vector product of $\overrightarrow{\widehat{a}}$ $=%
\overrightarrow{a}+\varepsilon \overrightarrow{a^{\ast }}\in \mathbb{D}^{3}$
and $\overrightarrow{\widehat{b}}$ $=\overrightarrow{b}+\varepsilon 
\overrightarrow{b^{\ast }}$ $\in \mathbb{D}^{3}$ are given by 
\begin{eqnarray*}
\left\langle \overrightarrow{\widehat{a}},\overrightarrow{\widehat{b}}%
\right\rangle &=&\left\langle \overrightarrow{a},\overrightarrow{b}%
\right\rangle +\varepsilon \left( \left\langle \overrightarrow{a},%
\overrightarrow{b^{\ast }}\right\rangle +\left\langle \overrightarrow{%
a^{\ast }},\overrightarrow{b}\right\rangle \right) \\
\overrightarrow{\widehat{a}}\times \overrightarrow{\widehat{b}} &=&\left( 
\widehat{a}_{2}\widehat{b}_{3}-\widehat{a}_{3}\widehat{b}_{2}\text{ },%
\widehat{a}_{3}\widehat{b}_{1}-\widehat{a}_{1}\widehat{b}_{3}\text{ , }%
\widehat{a}_{1}\widehat{b}_{2}-\widehat{a}_{2}\widehat{b}_{1}\right)
\end{eqnarray*}%
where $\widehat{a}_{i}=a_{i}+\varepsilon a_{i}^{\ast }$ , $\widehat{b}%
_{i}=b_{i}+\varepsilon b_{i}^{\ast }\in \mathbb{D}$ $,$ $1\leq i\leq 3.$

The norm $\left\Vert \overrightarrow{\widehat{a}}\right\Vert $of $%
\overrightarrow{\widehat{a}}$ is defined by%
\begin{equation*}
\left\Vert \overrightarrow{\widehat{a}}\right\Vert =\sqrt{\left\langle 
\overrightarrow{\widehat{a}},\overrightarrow{\widehat{a}}\right\rangle }%
=\left\Vert \overrightarrow{a}\right\Vert +\varepsilon \frac{\left\langle 
\overrightarrow{a},\overrightarrow{a^{\ast }}\right\rangle }{\left\Vert 
\overrightarrow{a}\right\Vert }
\end{equation*}%
where $a\neq 0.$

If the norm of $\overrightarrow{\widehat{a}}$ is 1, then $\overrightarrow{%
\widehat{a}}$ is called dual unit vector.

Let%
\begin{equation*}
\begin{array}{ccc}
\widehat{\alpha }:I\subset \mathbb{D} & \longrightarrow & \mathbb{D}^{3}%
\phantom{= V_{i}(s); \quad 1\leq i\leq n-1} \\ 
\phantom{\alpha : }\lambda & \longrightarrow & \overrightarrow{\widehat{%
\alpha }}(\lambda )=\overrightarrow{\alpha }(\lambda )+\varepsilon 
\overrightarrow{\alpha ^{\ast }}(\lambda )%
\end{array}%
\end{equation*}%
be a dual space curve with differentiable vectors $\overrightarrow{\alpha }%
(\lambda )$ and $\overrightarrow{\alpha ^{\ast }}(\lambda )$. The dual
arc-length parameter of $\overrightarrow{\widehat{\alpha }}(\lambda )$ is
defined as 
\begin{equation*}
s=\dint\limits_{t_{1}}^{t}\left\Vert \overrightarrow{\widehat{\alpha }}%
(\lambda )^{\prime }\right\Vert d\lambda .
\end{equation*}

Now we will give dual Frenet vectors of the dual curve 
\begin{equation*}
\begin{array}{ccc}
\widehat{\alpha }:I\subset \mathbb{D} & \longrightarrow & \mathbb{D}^{3}%
\phantom{= V_{i}(s); \quad 1\leq i\leq n-1} \\ 
\phantom{\alpha : }s & \longrightarrow & \overrightarrow{\widehat{\alpha }}%
(s)=\overrightarrow{\alpha }(s)+\varepsilon \overrightarrow{\alpha ^{\ast }}%
(s)%
\end{array}%
\end{equation*}%
with the dual arc-length parameter $s.$ Then 
\begin{equation*}
\frac{d\overrightarrow{\widehat{\alpha }}}{d\widehat{s}}=\frac{d%
\overrightarrow{\widehat{\alpha }}}{ds}.\frac{ds}{d\widehat{s}}=%
\overrightarrow{\widehat{T}}
\end{equation*}%
is called the unit tangent vector of $\overrightarrow{\widehat{\alpha }}(s)$%
. The norm of the vector $\frac{d\overrightarrow{\widehat{T}}}{d\widehat{s}}$
which is given by%
\begin{equation*}
\frac{d\overrightarrow{\widehat{T}}}{d\widehat{s}}=\frac{d\overrightarrow{%
\widehat{T}}}{ds}.\frac{ds}{d\widehat{s}}=\frac{d^{2}\overrightarrow{%
\widehat{\alpha }}}{d\widehat{s}^{2}}=\widehat{\kappa }\overrightarrow{%
\widehat{N}}
\end{equation*}%
is called curvature function of $\overrightarrow{\widehat{\alpha }}(s)$.
Here $\widehat{\kappa }:I\longrightarrow \mathbb{D}$ is nowhere pure-dual.
Then the unit principal normal vector of $\overrightarrow{\widehat{\alpha }}%
(s)$ is defined as%
\begin{equation*}
\overrightarrow{\widehat{N}}=\frac{1}{\widehat{\kappa }}.\frac{d%
\overrightarrow{\widehat{T}}}{d\widehat{s}}
\end{equation*}%
The vector $\overrightarrow{\widehat{B}}=\overrightarrow{\widehat{T}}\times 
\overrightarrow{\widehat{N}}$ is called the binormal vector of $%
\overrightarrow{\widehat{\alpha }}(s)$. Also we call the vectors $%
\overrightarrow{\widehat{T}},\overrightarrow{\widehat{N}},\overrightarrow{%
\widehat{B}}$ Frenet trihedron of $\overrightarrow{\widehat{\alpha }}(s)$ at
the point $\widehat{\alpha }(s)$. The derivatives of dual Frenet vectors $%
\overrightarrow{\widehat{T}},\overrightarrow{\widehat{N}},\overrightarrow{%
\widehat{B}}$ can be written in matrix form as 
\begin{equation*}
\left[ 
\begin{array}{l}
\overrightarrow{\widehat{T}}^{\prime } \\ 
\overrightarrow{\widehat{N}}^{\prime } \\ 
\overrightarrow{\widehat{B}}^{\prime }%
\end{array}%
\right] =\left[ 
\begin{array}{lll}
0 & \widehat{\kappa } & 0 \\ 
-\widehat{\kappa } & 0 & \widehat{\tau } \\ 
0 & -\widehat{\tau } & 0%
\end{array}%
\right] \left[ 
\begin{array}{l}
\overrightarrow{\widehat{T}} \\ 
\overrightarrow{\widehat{N}} \\ 
\overrightarrow{\widehat{B}}%
\end{array}%
\right]
\end{equation*}%
which are called Frenet formulas \cite{Kos}. The function $\widehat{\tau }%
:I\longrightarrow \mathbb{D}$ such that $\frac{d\overrightarrow{\widehat{B}}%
}{d\widehat{s}}=-\widehat{\tau }\overrightarrow{\widehat{N}}$ is called the
torsion of $\overrightarrow{\widehat{\alpha }}(s)$ which is nowhere
pure-dual.

For a general parameter $t$ of a dual space curve $\overrightarrow{\widehat{%
\alpha }}$ , the curvature and torsion of $\overrightarrow{\widehat{\alpha }}
$ can be calculated as ;%
\begin{equation*}
\widehat{\kappa }=\frac{\left\Vert \widehat{\alpha }^{^{\prime }}\times 
\widehat{\alpha }^{^{\prime \prime }}\right\Vert }{\left\Vert \widehat{%
\alpha }^{^{\prime }}\right\Vert ^{3}}\text{ \ \ \ \ , \ \ }\widehat{\tau }=%
\frac{det\text{ }(\text{ }\widehat{\alpha }^{^{\prime }},\widehat{\alpha }%
^{^{\prime \prime }},\widehat{\alpha }^{^{\prime \prime \prime }})}{%
\left\Vert \widehat{\alpha }^{^{\prime }}\times \widehat{\alpha }^{^{\prime
\prime }}\right\Vert }.
\end{equation*}

\section{Bertrand curves in $\mathbb{D}^{3}$}

In this section, we define Bertrand curves in dual space $\mathbb{D}^{3}$
and give characterizations and theorems for these curves.

\textbf{Definition 1. }Let $\mathbb{D}^{3}$ be the dual space with standart
inner product $\left\langle ,\right\rangle $ and $\overrightarrow{\widehat{%
\alpha }}$ and $\overrightarrow{\widehat{\beta }}$ be the dual space curves.
\ If there exists a corresponding relationship between the dual space curves 
$\overrightarrow{\widehat{\alpha }}$ and $\overrightarrow{\widehat{\beta }}$
so that the principal normal vectors of $\overrightarrow{\widehat{\alpha }}$
and $\overrightarrow{\widehat{\beta }}$ are linear dependent to each other
at the corresponding points of the dual curves, then $\overrightarrow{%
\widehat{\alpha }}$ and $\overrightarrow{\widehat{\beta }}$ are called
Bertrand curves\ in $\mathbb{D}^{3}.$

Let the curves $\overrightarrow{\widehat{\alpha }}$ and $\overrightarrow{%
\widehat{\beta }}$ be Bertrand curves in $\mathbb{D}^{3}$, parameterized by
their arc-length $s$ and $\widehat{s}$, respectively. Let $\left\{ 
\overrightarrow{\widehat{T}}(s)\text{ , }\overrightarrow{\widehat{N}}(s)%
\text{ , }\overrightarrow{\widehat{B}}(s)\right\} $ indicate the unit Frenet
frame along $\overrightarrow{\widehat{\alpha }}$ and $\left\{ 
\overrightarrow{\widehat{T^{%
{{}^\circ}%
}}}(s)\text{ , }\overrightarrow{\widehat{N^{%
{{}^\circ}%
}}}(s)\text{ , }\overrightarrow{\widehat{B^{%
{{}^\circ}%
}}}(s)\right\} $ the unit Frenet frame along $\overrightarrow{\widehat{\beta 
}}.$ Also $\widehat{\kappa }(s)=\kappa (s)+\varepsilon \kappa ^{\ast }(s)$
and $\widehat{\tau }(s)=\tau (s)+\varepsilon \tau ^{\ast }(s)$ are the
curvature and torsion of $\overrightarrow{\widehat{\alpha }}$ ,
respectively. Similarly, $\widehat{\kappa ^{%
{{}^\circ}%
}}(s)$ $=\kappa ^{%
{{}^\circ}%
}(s)+\varepsilon \kappa ^{%
{{}^\circ}%
\ast }(s)$ and $\widehat{\tau ^{%
{{}^\circ}%
}}(s)=\tau ^{%
{{}^\circ}%
}(s)+\varepsilon \tau ^{%
{{}^\circ}%
\ast }(s)$ are the curvature and torsion of $\overrightarrow{\widehat{\beta }%
}$ , respectively.

In the following theorems, we obtain the characterizations of a dual
Bertrand curve.

\begin{theorem}
Let\textbf{\ }$\overrightarrow{\widehat{\alpha }}$ and $\overrightarrow{%
\widehat{\beta }}$ be two curves in $\mathbb{D}^{3}.$ If $\overrightarrow{%
\widehat{\alpha }}$ and $\overrightarrow{\widehat{\beta }}$ are Bertrand
curves, then 
\begin{equation*}
d\left( \overrightarrow{\widehat{\alpha }}(s),\overrightarrow{\widehat{\beta 
}}(s)\right) =\widehat{c}
\end{equation*}%
where $s\in I\subset \mathbb{D}$ and $\widehat{c}\in \mathbb{D}$(constant).
\end{theorem}

\begin{proof}
Let $\overrightarrow{\widehat{\alpha }}$ and $\overrightarrow{\widehat{\beta 
}}$ be Bertrand curves.%
\begin{equation*}
\FRAME{itbpF}{3.5916in}{2.9507in}{0in}{}{}{Figure}{\special{language
"Scientific Word";type "GRAPHIC";maintain-aspect-ratio TRUE;display
"USEDEF";valid_file "T";width 3.5916in;height 2.9507in;depth
0in;original-width 4.9268in;original-height 4.0413in;cropleft "0";croptop
"1";cropright "1";cropbottom "0";tempfilename
'MM0J4L01.wmf';tempfile-properties "XPR";}}
\end{equation*}%
If $\overrightarrow{\widehat{\alpha }}$ and $\overrightarrow{\widehat{\beta }%
}$ are Bertrand curves\ , then we write from Figure 3.1%
\begin{equation}
\overrightarrow{\widehat{\beta }}(s)=\overrightarrow{\widehat{\alpha }}(s)+%
\widehat{\lambda }(s)\overrightarrow{\widehat{N}}(s)
\end{equation}

for the vectors of $\overrightarrow{\widehat{\alpha }}$ and $\overrightarrow{%
\widehat{\beta }}.$

If we differentiate the equation (3.1) with respect to $s$ and use the
Frenet equations, we get 
\begin{equation}
\frac{d\widehat{s}}{ds}.\overrightarrow{\widehat{T^{%
{{}^\circ}%
}}}(s)=\left( 1-\widehat{\lambda }(s)\widehat{\kappa }(s)\right) .%
\overrightarrow{\widehat{T}}(s)+\widehat{\lambda }^{^{\prime }}(s).%
\overrightarrow{\widehat{N}}(s)+\widehat{\lambda }(s).\widehat{\tau }(s).%
\overrightarrow{\widehat{B}}(s).
\end{equation}

If we take the inner product of the equation (3.2) with $\overrightarrow{%
\widehat{N}}(s)$ both sides, 
\begin{equation*}
0=\widehat{\lambda }^{^{\prime }}(s)
\end{equation*}

is found.Then%
\begin{equation*}
\widehat{\lambda }(s)=\widehat{c}
\end{equation*}

where $\widehat{c}\in \mathbb{D}$(constant). If we use%
\begin{equation*}
d\left( \overrightarrow{\widehat{\alpha }}(s),\overrightarrow{\widehat{\beta 
}}(s)\right) =\left\Vert \overrightarrow{\widehat{\beta }}(s)-%
\overrightarrow{\widehat{\alpha }}(s)\right\Vert
\end{equation*}

and the equation (3.1), we obtain 
\begin{equation*}
d\left( \overrightarrow{\widehat{\alpha }}(s),\overrightarrow{\widehat{\beta 
}}(s)\right) =\widehat{c}
\end{equation*}

where $s\in I\subset \mathbb{D}$ and $\widehat{c}\in \mathbb{D}$(constant).

Namely we mean that the distance between the corresponding points of the
dual Bertrand curves is constant.
\end{proof}

\begin{theorem}
Let\textbf{\ }$\overrightarrow{\widehat{\alpha }}$ and $\overrightarrow{%
\widehat{\beta }}$ be two curves in $\mathbb{D}^{3}.$ If $\overrightarrow{%
\widehat{\alpha }}$ and $\overrightarrow{\widehat{\beta }}$ are Bertrand
curves, then the dual angle between the tangent vectors at the corresponding
points of the dual Bertrand curves is constant.
\end{theorem}

\begin{proof}
Let\textbf{\ }$\overrightarrow{\widehat{\alpha }}$ and $\overrightarrow{%
\widehat{\beta }}$ be two Bertrand curves in $\mathbb{D}^{3}.$

If the dual angle between the tangent vectors$\overrightarrow{\widehat{\text{
}T}}(s)$ and $\overrightarrow{\widehat{T^{%
{{}^\circ}%
}}}(s)$ at the corresponding points of $\overrightarrow{\widehat{\alpha }}$
and $\overrightarrow{\widehat{\beta }}$ is%
\begin{equation*}
\phi =\varphi +\varepsilon \varphi ^{\ast }\in \mathbb{D}\text{,}
\end{equation*}

then we write 
\begin{equation}
\overrightarrow{\widehat{T^{%
{{}^\circ}%
}}}(s)=\cos \phi \overrightarrow{\widehat{\text{ }T}}(s)+\sin \phi 
\overrightarrow{\widehat{\text{ }B}}(s).
\end{equation}

If we differentiate the equation (3.3), we get 
\begin{equation}
\widehat{\kappa ^{%
{{}^\circ}%
}}(s)\overrightarrow{\widehat{N^{%
{{}^\circ}%
}}}(s)\frac{d\widehat{s}}{ds}=\frac{d\cos \phi }{ds}\overrightarrow{\widehat{%
\text{ }T}}(s)+\left( \widehat{\kappa }(s)\cos \phi -\widehat{\tau }(s)\sin
\phi \right) \overrightarrow{\widehat{N}}(s)+\frac{d\sin \phi }{ds}%
\overrightarrow{\widehat{B}}(s).
\end{equation}

If we take the inner product of the equation (3.4) with $\overrightarrow{%
\widehat{T}}(s)$ both sides and we use the Frenet equations, we have

\begin{equation}
\frac{d\cos \phi }{ds}=0.
\end{equation}

So 
\begin{equation*}
\cos \phi =\text{constant}
\end{equation*}%
is found where $\phi =\varphi +\varepsilon \varphi ^{\ast }\in \mathbb{D}$.

This completes the proof.
\end{proof}

\begin{theorem}
Let\textbf{\ }$\overrightarrow{\widehat{\alpha }}$ and $\overrightarrow{%
\widehat{\beta }}$ be two curves in $\mathbb{D}^{3}.$If $\overrightarrow{%
\widehat{\alpha }}$ and $\overrightarrow{\widehat{\beta }}$ are Bertrand
curves , $\widehat{\kappa }(s)$ and $\widehat{\tau }(s)$ are the curvature
and torsion of $\overrightarrow{\widehat{\alpha }}$ \ $,$ $\widehat{\kappa ^{%
{{}^\circ}%
}}(s)$ and $\widehat{\tau ^{%
{{}^\circ}%
}}(s)$ are the curvature and torsion of $\overrightarrow{\widehat{\beta }}$\
, respectively, then%
\begin{equation*}
\widehat{\lambda }.\widehat{\kappa }(s)+\widehat{\mu }.\widehat{\tau }(s)=1
\end{equation*}%
where $\widehat{\lambda }$,$\widehat{\mu }$ $\in \mathbb{D}$ are constant.
\end{theorem}

\begin{proof}
Let $\overrightarrow{\widehat{\alpha }}$ and $\overrightarrow{\widehat{\beta 
}}$ are Bertrand curves. If the dual angle between the tangent vectors$%
\overrightarrow{\widehat{\text{ }T}}(s)$ and $\overrightarrow{\widehat{T^{%
{{}^\circ}%
}}}(s)$ at the corresponding points of $\overrightarrow{\widehat{\alpha }}$
and $\overrightarrow{\widehat{\beta }}$ is 
\begin{equation*}
\phi =\varphi +\varepsilon \varphi ^{\ast }\in \mathbb{D}\text{,}
\end{equation*}%
then from the previous proof we have 
\begin{equation}
\overrightarrow{\widehat{T^{%
{{}^\circ}%
}}}(s)=\cos \phi \overrightarrow{\widehat{\text{ }T}}(s)+\sin \phi 
\overrightarrow{\widehat{\text{ }B}}(s).
\end{equation}

From the equation (3.2) we write%
\begin{equation}
\frac{d\widehat{s}}{ds}.\overrightarrow{\widehat{T^{%
{{}^\circ}%
}}}(s)=\left( 1-\widehat{\lambda }.\widehat{\kappa }(s)\right) .%
\overrightarrow{\widehat{T}}(s)+\widehat{\lambda }.\widehat{\tau }(s).%
\overrightarrow{\widehat{B}}(s).
\end{equation}

In above equations, if we take into account 
\begin{equation*}
\frac{d\widehat{s}}{ds}=\widehat{a}\text{ (constant)}
\end{equation*}%
then we get 
\begin{equation}
1-\widehat{\lambda }.\widehat{\kappa }(s)=\cot \phi .\widehat{\lambda }.%
\widehat{\tau }(s).
\end{equation}

We can write 
\begin{equation}
\widehat{\mu }=\widehat{\lambda }.\cot \phi \text{ }(\text{constant}).
\end{equation}

Finally, from the equations (3.8) and (3.9), we find 
\begin{equation*}
\widehat{\lambda }.\widehat{\kappa }(s)+\widehat{\mu }.\widehat{\tau }(s)=1.
\end{equation*}
\end{proof}

\begin{theorem}
Let $\overrightarrow{\widehat{\alpha }}$ be a plane curve in $\mathbb{D}%
^{3}. $ If $\overrightarrow{\widehat{\beta }}$ and $\overrightarrow{\widehat{%
\gamma }}$ are the involutes of $\overrightarrow{\widehat{\alpha }}$ , then $%
\overrightarrow{\widehat{\beta }}$ and $\overrightarrow{\widehat{\gamma }}$
are Bertrand curves in $\mathbb{D}^{3}$ .
\end{theorem}

\begin{proof}
Let $\overrightarrow{\widehat{\beta }}$ and $\overrightarrow{\widehat{\gamma 
}}$ be the involutes of $\overrightarrow{\widehat{\alpha }}.$%
\begin{equation*}
\FRAME{itbpF}{6.2388in}{2.7778in}{0in}{}{}{Figure}{\special{language
"Scientific Word";type "GRAPHIC";maintain-aspect-ratio TRUE;display
"USEDEF";valid_file "T";width 6.2388in;height 2.7778in;depth
0in;original-width 6.2811in;original-height 2.7812in;cropleft "0";croptop
"1";cropright "1";cropbottom "0";tempfilename
'MM0L7O02.wmf';tempfile-properties "XPR";}}
\end{equation*}

It can be written from \cite{Guv2}

\begin{equation}
\widehat{\tau ^{%
{{}^\circ}%
}}(s)=\frac{\widehat{\kappa }(s).\widehat{\tau }^{^{\prime }}(s)-\widehat{%
\kappa }^{^{\prime }}(s).\widehat{\tau }(s)}{\widehat{\kappa }(s).(\widehat{c%
}-s).\left( \widehat{\kappa }^{2}(s)+\widehat{\tau }^{2}(s)\right) }
\end{equation}%
where $\widehat{\kappa }(s)$, $\widehat{\tau }(s)$ and $\widehat{\kappa ^{%
{{}^\circ}%
}}(s)$ , $\widehat{\tau ^{%
{{}^\circ}%
}}(s)$ are the curvature and torsion of $\overrightarrow{\widehat{\alpha }}$
and $\overrightarrow{\widehat{\beta }}$ , respectively, and $\widehat{c}\in 
\mathbb{D(}$constant$\mathbb{)}$.

If $\overrightarrow{\widehat{\alpha }}$ is a plane curve ,%
\begin{equation}
\widehat{\tau }(s)=0.
\end{equation}

From the equation (3.10) and (3.11), we have 
\begin{equation*}
\widehat{\tau ^{%
{{}^\circ}%
}}(s)=0.
\end{equation*}

So $\overrightarrow{\widehat{\beta }}$ is a plane curve. Similarly, $%
\overrightarrow{\widehat{\gamma }}$ is a plane curve. Consequently, the
principal normal vectors of $\overrightarrow{\widehat{\beta }}$ and $%
\overrightarrow{\widehat{\gamma }}$ are linear dependent to each other at
the corresponding points of the dual curves. In that case $\overrightarrow{%
\widehat{\beta }}$ and $\overrightarrow{\widehat{\gamma }}$ curves are
Bertrand curves in $\mathbb{D}^{3}.$
\end{proof}

\bigskip

\.{I}lkay Arslan G\"{u}ven

University of Gaziantep

Department of Mathematics

\c{S}ehitkamil, 27310, Gaziantep, Turkey

E-mail: iarslan@gantep.edu.tr, ilkayarslan81@hotmail.com

\bigskip

\.{I}pek A\u{g}ao\u{g}lu

University of Gaziantep

Department of Mathematics

\c{S}ehitkamil, 27310, Gaziantep, Turkey

E-mail: agaogluipek@gmail.com

\end{document}